\documentclass[pdflatex,sn-mathphys-num]{sn-jnl}

\usepackage{graphicx}
\usepackage{multirow}
\usepackage{amsmath,amssymb,amsfonts}
\usepackage{amsthm}
\usepackage{mathrsfs}
\usepackage[title]{appendix}
\usepackage{xcolor}
\usepackage{subcaption}
\usepackage{textcomp}
\usepackage{manyfoot}
\usepackage{booktabs}
\usepackage{algorithm}
\usepackage{algorithmicx}
\usepackage{algpseudocode}
\usepackage{listings}

\usepackage{geometry}
\usepackage{tikz}
\usetikzlibrary{calc,shadows.blur,positioning,arrows.meta, matrix, fit}

\newcommand\QueensChessBoardp[4]{
    \def\N{#1}
    \def\positions{#2}
    \def\highlights{#3} 
    \begin{tikzpicture}[scale=#4, every node/.style={minimum size=1cm * #4}]
        \foreach \i in {1,...,\N}{
            \foreach \j in {1,...,\N}{
                \pgfmathparse{mod(\i+\j,2) ? "brown!80" : "cream!70"}
                \edef\color{\pgfmathresult}
                \path[fill=\color] (\i, \j) rectangle + (1, 1);
            }
        }
        
        \foreach \x/\y in \highlights{
            \path[fill=red!50] (\x, \N -\y +1) rectangle + (1, 1); 
        }
        
        \foreach \x in {1,...,\N}{
            \node at (\x+.5, .5) {\scriptsize\x};
            \node at (.5, \N - \x + 1.5) {\scriptsize\x};
        }
        
        \foreach \x/\y in \positions{
            \node at (\x+.5, \N -\y+1.5) {\includegraphics[width=0.35cm]{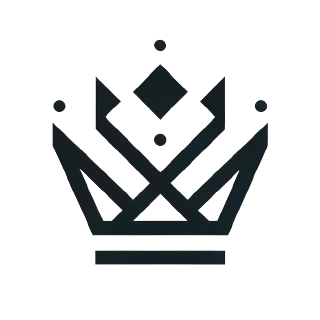}};
        }
    \end{tikzpicture}
}

\makeatletter
\tikzset{
  a/.style={
    double=green!10,
    double distance=2pt,
    draw=teal!50,
  },
  foo /.tip={Stealth[inset=0pt, length=4pt, width=6pt, fill=green!10]},
  pics/arc arrow/.style args={#1:#2:#3}{
    code={
      \draw[a, -foo] (0, 0) arc (#1:#2:#3) coordinate (arc@temp);
      \path (arc@temp) ++(270+#2:4pt) ++(#2:3pt) coordinate (-a);
    }
  },
  pics/straight arrow/.style={
    code={
      \begin{scope}[rotate=#1]
        \draw[a, foo-foo] (0, -0.5) -- (0, 0.5);
        \fill[green!10] (-1pt, 0.5cm-4.1pt) rectangle (1pt, 0.5cm-3.9pt+\pgflinewidth);
        \fill[green!10] (1pt, -0.5cm+4.1pt) rectangle (-1pt, -0.5cm+3.9pt-\pgflinewidth);
        \coordinate (-a) at (3pt, -0.5cm + 4pt);
      \end{scope}
    }
  },
  mydash/.style={dash pattern=on 2mm off 0.5mm}
}
\makeatother

\newcommand{\RotationArrow}{%
    \begin{tikzpicture}[baseline={([yshift=-11ex]current bounding box.center)}]
        \draw[->, >=Stealth, line width=0.8pt] (0,0) -- (1.2,0) node[midway, above] {\scriptsize $r$};
    \end{tikzpicture}
}

\definecolor{cream}{RGB}{255,253,208}
\colorlet{BlueC}{blue!80}
\colorlet{RedC}{red!80}
\colorlet{YellowC}{yellow!90!orange}
\colorlet{GreenC}{green!65!black}

\newcommand{\ULSwatch}[1]{
  {\setlength{\fboxsep}{0pt}
   \fcolorbox{#1!70!black}{#1}{\rule{0pt}{1.8ex}\hspace{1.8ex}}}
}

\newcommand{\ULHLCellInset}[7]{
  \path[
    fill=#4, fill opacity=#5,
    draw=#4!70!black, draw opacity=0.95,
    line width=#7
  ]
    ({#2 + #6},{#1 - #3 + 1 + #6}) rectangle ++({1 - 2*#6},{1 - 2*#6});
}

\newcommand{\QueensUlDiagAntiPairsTwelve}{
\begin{tikzpicture}[x=0.75cm,y=0.75cm,scale=0.95, every node/.style={inner sep=0pt}]
  \def\N{12}
  \def\LightSq{cream!85}
  \def\DarkSq{brown!80!black}

  \foreach \x in {1,...,\N}{
    \foreach \y in {1,...,\N}{
      \pgfmathtruncatemacro{\parity}{mod(\x+\y,2)}
      \ifnum\parity=0
        \path[fill=\LightSq, draw=black!12, line width=0.15pt] (\x,\y) rectangle ++(1,1);
      \else
        \path[fill=\DarkSq,  draw=black!12, line width=0.15pt] (\x,\y) rectangle ++(1,1);
      \fi
    }
  }

  \foreach \c/\r in {4/1,5/2,6/3,7/4,8/5,9/6,10/7,11/8,12/9}{\ULHLCellInset{\N}{\c}{\r}{BlueC}{0.92}{0.02}{1.05pt}}
  \foreach \c/\r in {9/1,8/2,7/3,6/4,5/5,4/6,3/7,2/8,1/9}{\ULHLCellInset{\N}{\c}{\r}{BlueC}{0.92}{0.02}{1.05pt}}

  \foreach \c/\r in {9/1,10/2,11/3,12/4}{\ULHLCellInset{\N}{\c}{\r}{RedC}{0.92}{0.08}{1.05pt}}
  \foreach \c/\r in {4/1,3/2,2/3,1/4}{\ULHLCellInset{\N}{\c}{\r}{RedC}{0.92}{0.08}{1.05pt}}

  \foreach \c/\r in {1/9,2/10,3/11,4/12}{\ULHLCellInset{\N}{\c}{\r}{YellowC}{0.92}{0.14}{1.05pt}}
  \foreach \c/\r in {12/9,11/10,10/11,9/12}{\ULHLCellInset{\N}{\c}{\r}{YellowC}{0.92}{0.14}{1.05pt}}

  \foreach \c/\r in {1/4,2/5,3/6,4/7,5/8,6/9,7/10,8/11,9/12}{\ULHLCellInset{\N}{\c}{\r}{GreenC}{0.92}{0.20}{1.05pt}}
  \foreach \c/\r in {12/4,11/5,10/6,9/7,8/8,7/9,6/10,5/11,4/12}{\ULHLCellInset{\N}{\c}{\r}{GreenC}{0.92}{0.20}{1.05pt}}

  \foreach \c/\r in {4/1,12/4,9/12,1/9}{
    \node at ({\c+0.5},{\N-\r+1.5}) {\includegraphics[width=0.55cm]{./images/Queen-Logo-GPT-removebg-preview-cropped.png}};
  }

  \foreach \i in {1,...,\N}{
    \node[font=\scriptsize] at (\i+0.5,0.35) {\i};
    \node[font=\scriptsize] at (0.35,\N-\i+1.5) {\i};
  }

  \node[
    anchor=north west,
    draw=black!25,
    fill=white,
    rounded corners=2pt,
    inner sep=5pt,
    align=left,
    font=\scriptsize
  ] at (\N+1.6,\N+1.05) {
    \textbf{Diagonal Pairing}\\[2pt]
    \begin{tabular}{@{}l l@{}}
      \ULSwatch{BlueC}   & $D^{-}_{l-1+n} \cup D^{+}_{n+1-l}$\\
      \ULSwatch{RedC}    & $D^{-}_{2n-l} \cup D^{+}_{l}$\\
      \ULSwatch{YellowC} & $D^{-}_{l} \cup D^{+}_{2n - l}$\\
      \ULSwatch{GreenC}  & $D^{-}_{n+1-l} \cup D^{+}_{n+l-1}$\\
    \end{tabular}
  };

\end{tikzpicture}
}

\tikzset{
  config/.style={
    draw,
    rectangle,
    rounded corners=2pt,
    minimum width=1.2cm,
    minimum height=0.7cm,
    fill=blue!5
  },
  repbox/.style={
    draw=red!70,
    ultra thick,
    rounded corners=3pt
  }
}

\theoremstyle{thmstyleone}
\newtheorem{theorem}{Theorem}
\newtheorem{proposition}[theorem]{Proposition}
\newtheorem{lemma}[theorem]{Lemma}

\theoremstyle{thmstyletwo}

\theoremstyle{thmstylethree}

\begin{document}

\title[$Q(n)$ is divisible by $4$]{$Q(n)$ is divisible by $4$}

\author*[1]{\fnm{Hugo Møller} \sur{Nielsen}}\email{hugomn2002@gmail.com}

\affil[1]{\orgdiv{Department of Applied Mathematics and Computer Science}, \orgname{Technical University of Denmark (DTU)}, 
  \orgaddress{\city{Kgs. Lyngby}, \postcode{2800}, \country{Denmark}}}

\abstract{
    We consider the classical $n$-queens problem, which asks how many ways 
    one can place $n$ mutually non-attacking queens on an $n \times n$ chessboard. 
    We prove that the total number of solutions to the $n$-queens problem 
    $Q(n)$ is divisible by 4 whenever $n \ge 6$.
}


\maketitle

\section{Introduction}\label{introduction}
An $n$-queens configuration is a placement of $n$ mutually non-attacking queens on an $n \times n$ chessboard. That is, no 
two queens are contained in the same row, column, or diagonal. The $n$-queens problem, as posed by Max Bezzel in 
1848 \cite{garcia2023nqueens}, asks: for a given board size $n$, how many distinct $n$-queens configurations exist. We 
denote this number by $Q(n)$. 

The problem has attracted attention from notable mathematicians including Gauss, P\'olya, and Lucas 
\cite{Pratt2019NQueens}. At first, the values of $Q(n)$ were computed by hand, or at least without digital assistance. 
These were the values of $Q(n)$ for $n \leq 13$, see \cite{BellStevens2009}. 
With the advent of computers, computing $Q(n)$ for larger $n$ became feasible; 
the current record being $Q(27) = 234907967154122528$, which required an immense amount of time and computational power to 
compute \cite{preusser_q27_2016}. 
Today the $n$-queens problem is often introduced at universities in the context of computation, where it 
serves as an example problem to be tackled by backtracking, constraint programming, or genetic algorithms, 
but theoretical results pertaining to $Q(n)$ remain an active area of research \cite{Simkin2023NQueens,Glock2022NQueens}.

Existence of solutions to the $n$-queens problem for $n \geq 4$ was shown using a constructive proof 
in 1874 by Pauls \cite{Pauls1874Damenproblem}. The simple result that $Q(n)$ is divisible by $2$ for all 
$n \geq 2$ has been known for at least as long \cite{oeisA000170, kraitchik1953mathematical}, 
a short proof of which will be provided in this work. More recently, 
the asymptotic behavior of $Q(n)$ was characterized and proven in 2021 to be of the form $Q(n) = ((1 + o(1))ne^{-\alpha})^n$
by Simkin \cite{Simkin2023NQueens}.

In addition to recalling a short proof that $Q(n)$ is even for all $n\ge 2$, we prove the stronger divisibility statement
in Theorem~\ref{thm:Qn_divby4_1}.
\begin{theorem} \label{thm:Qn_divby4_1}
    For $n \in \mathbb{N} \setminus \lbrace 1, 4, 5 \rbrace$ we have $Q(n) \equiv 0 \pmod{4}$.
\end{theorem}
We prove Theorem~\ref{thm:Qn_divby4_1} in Section~\ref{symmetry_of_Qn}. In the next section, Section~\ref{background}, we 
provide the tools and context necessary for the proof.

\section{Preliminaries}\label{background}
We briefly outline the notation used in this work and some basic properties of the $n$-queens problem.

\subsection{Representing $n$-queens configurations}
Throughout, $n$ denotes the board size. For $N \in \mathbb{N}$, write $[N] = \lbrace 1,2, \dots, N \rbrace$.
We represent the chessboard as the two-dimensional grid $[n] \times [n]$, using notation similar to that in \cite{Glock2022NQueens},
where rows, columns, diagonals and anti-diagonals are represented, respectively, by
\begin{align}
  R_{i} &= \lbrace (i,j) | j \in [n] \rbrace, \\
  C_{j} &= \lbrace (i,j) | i \in [n] \rbrace, \\
  D_{k}^{-} &= \lbrace (i,j) \in [n] \times [n] | j - i + n = k \rbrace, \\
  D_{k}^{+} &= \lbrace (i,j) \in [n] \times [n] | i + j - 1 = k \rbrace.
\end{align}
We identify the square in row $i \in [n]$ and column $j \in [n]$ with the coordinate $(i,j) \in [n] \times [n]$.
We define an \emph{$n$-queens configuration} to be a set $Q \subseteq [n] \times [n]$ of cardinality $|Q| = n$ such that 
\begin{align}
    &|Q \cap R_i| \leq 1 \;\;\; \forall i \in [n], &&|Q \cap C_j| \leq 1 \;\;\; \forall j \in [n], \label{eq:row_col_constraint} \\
    &|Q \cap D_{k}^{-}| \leq 1 \;\;\; \forall k \in [2n-1], &&|Q \cap D_{k}^{+}| \leq 1 \;\;\; \forall k \in [2n-1]. \label{eq:diag_constraint}
\end{align}
Given a set $S \subseteq [n] \times [n]$, we say that $U \subseteq [n] \times [n]$ is a \emph{completion of $S$} if 
$S \cup U$ is an $n$-queens configuration and $S \cap U = \emptyset$.

\subsection{Symmetries} \label{symmetries}
As the $n$-queens problem is based on placing queens on a square board, the dihedral group, $D_4$, naturally plays an important role in mathematical observations 
and efficient solution implementations of the $n$-queens problem. $D_4$ is the symmetry group of the square, namely the composition
of zero or more $90$ degree counter-clockwise rotations, $r$, with zero or more vertical mirroring operations, $s$, and with $e$ denoting the identity operation, 
see Figure~\ref{fig:geometric-relationships}. 
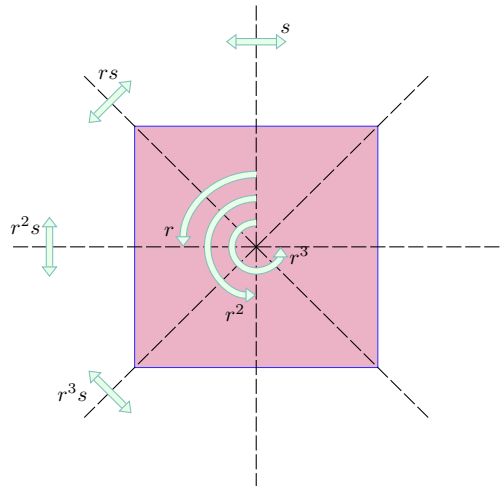
\begin{figure}[h!]
    \centering
    \begin{tikzpicture}[scale=0.8, transform shape]
        \path[draw=blue!80, fill=purple!30] (-2, -2) rectangle (2, 2);
        \foreach \i in {0, 45, ..., 315} {
          \draw[mydash] (0, 0) -- (\i:4);
        }
        \foreach \r/\t/\p [count=\i] in {90/s/above right, 135/rs/above left, 180/r^2s/left, 225/r^3s/below left} {
          \pic (arrow\i) at (\r:3.4cm) {straight arrow=\r};
          \node[\p, inner sep=1pt] at (arrow\i-a) {$\t$};
        }
        \foreach \x/\r/\t/\p [count=\i] in {0.4/360/r^3/right, 0.8/270/r^2/below left, 1.2/180/r/above left} {
          \pic (arc\i) at (0, \x) {arc arrow=90:\r:\x cm};
          \node[\p, inner sep=1pt] at (arc\i-a) {$\t$};
        }
    \end{tikzpicture}
    \caption{An illustrative diagram of the $D_4$ symmetry group.\protect\footnotemark}
    \label{fig:geometric-relationships}
\end{figure}
\footnotetext{Diagram inspired by the post. Accessed: 2024-02-17 (2020). \url{https://tex.stackexchange.com/questions/552589/square-rotational-and-reflection-symmetries}}

We define the group actions $e,r,s$ as functions on the power set $\mathcal{P}([n] \times [n])$ by 
\begin{align}
    eQ &= Q, \\
    sQ &= \lbrace (i, n + 1 - j) | (i,j) \in Q \rbrace, \\
    rQ &= \lbrace (n+1-j, i) | (i,j) \in Q \rbrace.
\end{align}
Note here that each symmetry $x \in D_4$ can be seen as a change in perspective of the observer of the chessboard. Changing 
the perspective of the observer clearly does not make two mutually non-attacking queens on a chessboard mutually attacking, 
hence making $xQ$ a valid $n$-queens configuration for each $n$-queens configuration $Q \subseteq [n] \times [n]$.
The $n$-queens configuration $Q$ is said to be \emph{fixed by the symmetry} $x \in D_4$ if $xQ = Q$, see Figure~\ref{fig:r-symmetry-5-queens}.
\begin{figure}[h!]
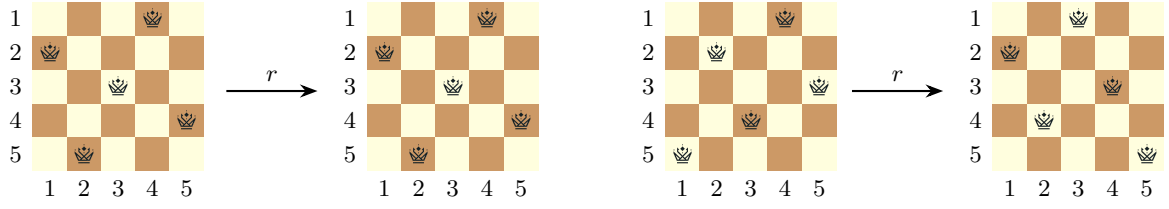

    \centering

    \QueensChessBoardp{5}{2/5,5/4,3/3,1/2,4/1}{}{0.45}
    \hspace{0.1em}\RotationArrow\hspace{-1em}
    \QueensChessBoardp{5}{2/5,5/4,3/3,1/2,4/1}{}{0.45}%
    \qquad
    \QueensChessBoardp{5}{1/5,3/4,5/3,2/2,4/1}{}{0.45}%
    \hspace{0.1em}\RotationArrow\hspace{-1em}
    \QueensChessBoardp{5}{5/5, 4/3, 3/1, 2/4, 1/2}{}{0.45}%

    \caption{Examples of 5-queens configurations with (left) and without (right) rotational symmetry under $r$, 
    the $90^\circ$ rotation of the $5\times 5$ board.}
    \label{fig:r-symmetry-5-queens}
\end{figure}

\begin{lemma}\label{lem:no-s-symmetry}
    There is no $n$-queens solution $Q$ that is fixed by one of the reflections $s, rs, r^2s$ or $r^3s$ for $n \geq 2$. 
\end{lemma}
\begin{proof}
    We have $Q(2) = Q(3) = 0$, and can manually check that Lemma~\ref{lem:no-s-symmetry} holds true for $n = 4$, 
    so consider now the case $n \geq 5$. Fix any $n$-queens configuration $Q$ and let $(i,j) \in Q$ be a queen with 
    $(i,j) \notin R_{\left\lfloor \frac{n+1}{2} \right\rfloor} \cup C_{\left\lfloor \frac{n+1}{2} \right\rfloor} \cup D_n^+ \cup D_n^-$, existence 
    of which is guaranteed by $n \geq 5$ and the constraints given in (\ref{eq:row_col_constraint}) and (\ref{eq:diag_constraint}), by virtue 
    of $Q$ being an $n$-queens configuration. Then we can break the proof into the following four cases:
    \begin{align}
        sQ &= Q \implies (i, n+1-j) \in sQ = Q \implies |Q \cap R_i| > 1, \\
        r^2sQ &= Q \implies (n+1-i, j) \in r^2sQ = Q \implies |Q \cap C_j| > 1, \\
        rsQ &= Q \implies (j, i) \in rsQ = Q \implies |Q \cap D_{j+i-1}^{+}| > 1, \\
        r^3sQ &= Q \implies (n+1-j, n+1-i) \in r^3sQ = Q \implies |Q \cap D_{n+j-i}^{-}| > 1,
    \end{align}
    each of which is a contradiction according to (\ref{eq:row_col_constraint}) or (\ref{eq:diag_constraint}).
\end{proof}
From Lemma~\ref{lem:no-s-symmetry} it easily follow that $Q(n)$ is even.
\begin{proposition} \label{cor:Qn_even}
    $Q(n)$ is even for all $n \geq 2$.
\end{proposition}
\begin{proof}
    Fix $n \geq 2$ and let $S_n$ denote the set of all $n$-queens configurations on the $n \times n$ board, so that
    $\lvert S_n \rvert = Q(n)$.
    By Lemma~\ref{lem:no-s-symmetry}, no configuration $Q \in S_n$ is fixed by the reflection $s$, that is, $sQ \neq Q$ for all
    $Q \in S_n$. Moreover, for two distinct configurations $Q_1, Q_2 \in S_n$ we have $sQ_1 \neq sQ_2$ 
    by injectivity of the map $Q \mapsto sQ$.
    Thus $S_n$ decomposes into disjoint orbits of the form $\{Q, sQ\}$, each containing exactly two distinct
    configurations. It follows that $\lvert S_n \rvert$ is even, and hence $Q(n)$ is even.
\end{proof}

As per Lemma~\ref{lem:no-s-symmetry}, an $n$-queens configuration can only possibly be fixed by one of the symmetry 
actions $e, r, r^2, r^3$. The identity $e$ trivially fixes all $n$-queens configurations, leaving us with the actions 
$r, r^2$ and $r^3$. If an $n$-queens configuration $Q$ is fixed under $r$, then it is fixed under $r^3$ since 
$r^3Q = r^2(rQ) = r^2Q = r(rQ) = rQ = Q$, and if $Q$ is fixed under $r^3$, then it is fixed under $r$ since 
$rQ = r(r^3Q) = r^4Q = eQ = Q$, hence the set of $n$-queens configurations that are fixed under $r$ is the same 
set of $n$-queens configurations that are fixed under $r^3$. 

Thus, we conclude that we can partition 
the set of all $Q(n)$ $n$-queens configurations into those solutions that are fixed under the $r$-symmetry, $F_r$, 
those solutions that are fixed under the $r^2$-symmetry, but not the $r$-symmetry, $F_{r^2 \setminus r}$, and 
those solutions that are only fixed by the identity symmetry operator $e$, $F_e$. By this partition, we must have that 
\begin{align} \label{eq:Qn_decomposition_Fx}
    Q(n) = |F_r| + |F_{r^2 \setminus r}| + |F_e|.
\end{align}
See Figure~\ref{fig:r-symmetry-5-queens} 
for an example in $F_r$ and $F_e$ respectively, for the case $n = 5$. Now define $C_x$ as a set of equivalence class representatives under $D_4$ 
of the solutions in $F_x$ for $x \in \lbrace r, r^2 \setminus r, e \rbrace$. 

The final result of this section, Theorem~\ref{thm:Qn_decomposition}, shows that the cardinality of the set of solutions fixed under $F_r$, $F_{r^2 \setminus r}$ and 
$F_e$ respectively are divisible by $2, 4$ and $8$. This result is crucial in the proof of Theorem~\ref{thm:Qn_divby4_1}, but is also of independent 
interest.

\begin{theorem}\label{thm:Qn_decomposition}
For $n \geq 2$ we have $|F_r| = 2|C_r|$, $|F_{r^2 \setminus r}| = 4|C_{r^2 \setminus r}|$, and $|F_e| = 8|C_e|$, and hence
\begin{align}
    Q(n) = 2|C_r| + 4|C_{r^2 \setminus r}| + 8|C_e|.
\end{align}
\end{theorem}

\begin{proof}
Fix $n \geq 2$. For an $n$-queens configuration $Q$, define its orbit and stabilizer (under the actions of $D_4$) by
\begin{align}
    \mathrm{Orb}(Q)  &= \{\, yQ : y \in D_4 \,\},\\
    \mathrm{Stab}(Q) &= \{\, y \in D_4 : yQ = Q \,\}.
\end{align}
First, for each $x \in \{r,\, r^2 \setminus r,\, e\}$ and each $Q \in C_x$, we have $\mathrm{Orb}(Q) \subseteq F_x$. 
Consequently,
\begin{align}
    F_x = \{\, yQ : Q \in C_x,\ y \in D_4 \,\}.
\end{align}
Therefore, using the decomposition of $F_x$ into $D_4$-orbits, we obtain
\begin{align}\label{eq:sizeof_Fx_1}
    |F_x| = \sum_{Q \in C_x} |\mathrm{Orb}(Q)|.
\end{align}

Next, by the definition of the sets $F_x$, the stabilizer size is constant on each $F_x$ (and hence on each $C_x$). More precisely:
\begin{align}
    Q \in F_r
    &\iff rQ = Q
    &&\Rightarrow\ \mathrm{Stab}(Q) = \{e, r, r^2, r^3\}
    &&\Rightarrow\ |\mathrm{Stab}(Q)| = 4, \\
    Q \in F_{r^2 \setminus r}
    &\iff r^2Q = Q \text{ but } rQ \neq Q
    &&\Rightarrow\ \mathrm{Stab}(Q) = \{e, r^2\}
    &&\Rightarrow\ |\mathrm{Stab}(Q)| = 2, \\
    Q \in F_e
    &\iff \text{$Q$ is fixed only by $e$}
    &&\Rightarrow\ \mathrm{Stab}(Q) = \{e\}
    &&\Rightarrow\ |\mathrm{Stab}(Q)| = 1.
\end{align}

Combining \eqref{eq:sizeof_Fx_1} with the orbit--stabilizer theorem yields, for each $x \in \{r,\, r^2 \setminus r,\, e\}$,
\begin{align}
    |F_x|
    = \sum_{Q \in C_x} |\mathrm{Orb}(Q)|
    = \sum_{Q \in C_x} \frac{|D_4|}{|\mathrm{Stab}(Q)|}
    = \frac{8}{|\mathrm{Stab}(Q)|}\,|C_x|,
\end{align}
where $|\mathrm{Stab}(Q)|$ is the constant stabilizer size for configurations in $C_x$.

Finally, using \eqref{eq:Qn_decomposition_Fx}, we obtain
\begin{align}
    Q(n)
    = |F_r| + |F_{r^2 \setminus r}| + |F_e|
    = 2|C_r| + 4|C_{r^2 \setminus r}| + 8|C_e|.
\end{align}
\end{proof}

\section{Proof of Theorem~\ref{thm:Qn_divby4_1}}\label{symmetry_of_Qn}
By Theorem~\ref{thm:Qn_decomposition}, it suffices to show that $|F_r| \equiv 0 \pmod{4}$ for all $n \ge 6$.

Consider any $n$-queens configuration $Q$ that is fixed by $r$. The action of $r$ partitions the queens of $Q$ into $r$-orbits.
Each orbit has size $4$, except possibly a single fixed point (the center square), which can occur only when $n$ is odd.
Hence $n=|Q|$ must be of the form $n=4k$ (even $n$) or $n=4k+1$ (odd $n$). In particular, if $n \equiv 2,3\pmod{4}$ then $F_r=\emptyset$.
Thus, by Theorem~\ref{thm:Qn_decomposition}, we must have the following partial result:

\begin{lemma}\label{lem:4k_4kp1_imply_Qn_div4}
    $Q(n) \equiv 0 \; (\textrm{mod} \; 4)$ when $n \equiv 2,3 \; (\textrm{mod} \; 4)$.
\end{lemma}
Now suppose $n = 4k$ or $n = 4k + 1$ for some $k \in \mathbb{N}_{\geq 2}$. Define for all $\ell \in [n]$ the border configurations
\begin{align}
    U_\ell = \begin{cases}
        \lbrace (1, \ell), ( \ell, n), (n, n+1- \ell), (n+1- \ell, 1) \rbrace, &n \equiv 0 \pmod{4}, \\
        \lbrace (1, \ell), ( \ell, n), (n, n+1- \ell), (n+1-\ell, 1) \rbrace \cup \left\lbrace \left(\frac{n+1}{2}, \frac{n+1}{2}\right) \right\rbrace, &n \equiv 1 \pmod{4}.
    \end{cases}
\end{align}
Any element $Q \in F_r$ will have $U_\ell \subseteq Q$ for some $\ell \in [n]$, by definition of $F_r$.
Denote by $Q_r^{(\ell)}$ the set of $r$-symmetric $n$-queen completions of $U_\ell$, so that
\begin{align}
    |F_r| = \sum_{\ell = 1}^{n} |Q_r^{(\ell)}|, \label{eq:Fr_decomp}
\end{align}
and define the border configuration complements $D^{(\ell)} = \lbrace Q \setminus U_\ell | Q \in Q_r^{(\ell)} \rbrace$, which 
by construction has $|Q_r^{(\ell)}| = |D^{(\ell)}|$ for all $\ell \in [n]$.
\begin{lemma} \label{lem:Dl_even}
    $|D^{(\ell)}|$ is even for all $\ell \in [n]$, when $n \geq 6$.
\end{lemma}
\begin{proof}
    Fix $\ell \in [n]$ and consider any queen $(i,j) \in [n] \times [n]$ which is mutually attacking with a queen in 
    $U_\ell$, then also its reflection under $s$, $(i, n+1-j)$, is mutually attacking with a queen in $U_\ell$. This 
    follows from $(i,j)$ being mutually attacking with a queen in $U_\ell$ along a row, column, diagonal or anti-diagonal 
    having
    \begin{align}
        &(i,j) \in R_{m} &&\iff i = m &&&\iff (i, n+1-j) \in R_{m}, \\
        &(i,j) \in C_{m} &&\iff j = m &&&\iff (i, n+1-j) \in C_{n+1-m}, \\
        &\forall m \in \lbrace \ell, n, n+1-\ell, 1 \rbrace,
    \end{align}
    and
    \begin{align}
        &(i,j) \in D_{m}^{-} \iff j-i+n = m \iff (n+1 -j) + i - 1 = 2n - m \iff (i, n+1-j) \in D_{2n - m}^{+}, \\
        &\forall m \in \lbrace \ell-1+n, 2n - \ell, 1-\ell+n, \ell \rbrace,
    \end{align}
    noting that $R_{\ell}, R_{n}, R_{n+1-\ell}, R_{1}$ are the rows of the non-center piece queens of $U_\ell$, while $C_{\ell}, C_{n}, C_{n+1-\ell}, C_{1}$ 
    are the non-centerpiece columns and 
    $D_{\ell-1+n}^{-}, D_{2n-\ell}^{-}, D_{1-\ell+n}^{-}, D_{\ell}^{-}, D_{n-\ell+1}^{+}, D_{\ell}^{+}, D_{n+\ell-1}^{+}, D_{2n-\ell}^{+}$ 
    are the non-centerpiece diagonals and anti-diagonals of $U_\ell$, see Figure~\ref{fig:ul-diag-antidiag-pairs}. 

    \begin{figure}[h!]
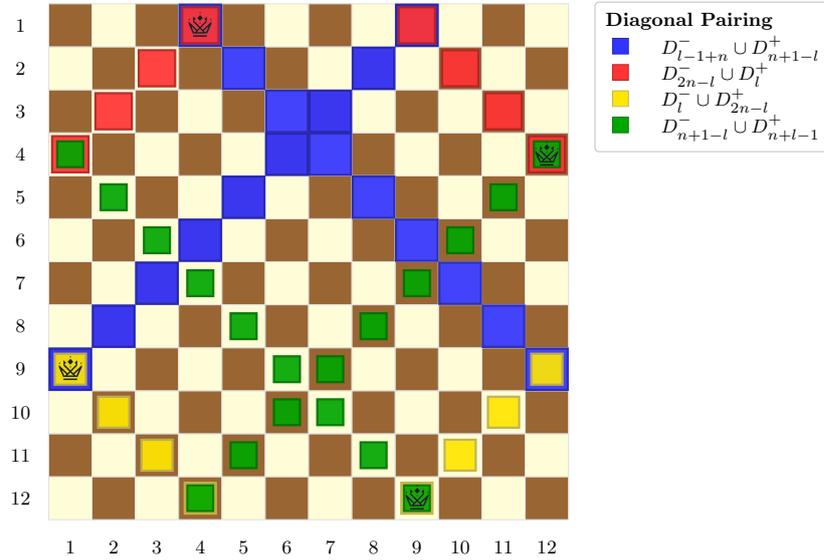

      \centering
      \scalebox{0.8}{\QueensUlDiagAntiPairsTwelve}
      \caption{The diagonals and anti-diagonals induced by the border queens of $U_\ell$ (shown here for $n=12$ and $\ell = 4$), 
      where each diagonal is paired with its corresponding anti-diagonal under the reflection $s$.}
      \label{fig:ul-diag-antidiag-pairs}
    \end{figure}

    For a possible center-piece in $U_\ell$, $(i,j)$ being mutually attacking with this center-piece clearly implies 
    $(i, n+1-j)$ mutually attacking with the center-piece.
    From this, we conclude $D \in D^{(\ell)} \iff sD \in D^{(\ell)}$. 
    Moreover, $sD \neq D$ for every $D \in D^{(\ell)}$. Indeed, if $sD = D$ then each queen $(i,j)\in D$ satisfies
    $(i,j)=(i,n+1-j)$, hence $j=(n+1)/2$. Thus every queen of $D$ lies in the center column.
    If $n$ is even this is impossible (no center column), and if $n$ is odd then $U_\ell$ already contains the center queen,
    so $D$ cannot place any queen in that column. In either case we get a contradiction.
    Therefore the involution $D \mapsto sD$ partitions $D^{(\ell)}$ into disjoint pairs $\{D,sD\}$, so $|D^{(\ell)}|$ is even.
\end{proof}
We have $|Q_r^{(\ell)}| = |Q_r^{(n+1-\ell)}|$ for each $\ell \in [n]$ as we have the involution given by 
$s : Q_r^{(\ell)} \to Q_r^{(n+1-\ell)}$. 
$Q_r^{(\frac{n+1}{2})} = \emptyset$ for $n \equiv 1 \pmod{4}$, as the queens in $U_\ell$ in this 
case are mutually attacking, therefore we can derive from (\ref{eq:Fr_decomp}) that
\begin{align} \label{eq:Fr_mod4_0}
    |F_r| = \sum_{\ell = 1}^{n} |Q_r^{(\ell)}| = 2\sum_{\ell=1}^{\lfloor \frac{n}{2} \rfloor} |Q_r^{(\ell)}|
    = 2\sum_{\ell=1}^{\lfloor \frac{n}{2} \rfloor} |D^{(\ell)}| \stackrel{\textrm{Lemma}~\ref{lem:Dl_even}}{\equiv} 0 \pmod{4}.
\end{align}
Thus completing the proof of Theorem~\ref{thm:Qn_divby4_1}.

Note that the requirement $k \geq 2$ in the proof above stems from the fact that for $n \in \lbrace 4,5 \rbrace$ we have
some $\ell \in [n]$ with $\emptyset \in D^{(\ell)}$ and hence contradicting $sD \neq D$ in the proof of Lemma~\ref{lem:Dl_even}.

\backmatter

\bmhead{Acknowledgments}
I am grateful to Jakob Lemvig and Christian Henriksen for their supervision, regular meetings,
feedback on methodology, and helpful comments.

\bibliography{sn-bibliography}

@book{garcia2023nqueens,
  title={The n-Queens Problem: An Activity Book},
  author={Garc\'{i}a S\'{a}nchez, Alicia},
  year={2023},
  publisher={University of Warwick, Mathematics Institute},
  address={Coventry, UK},
  series={Undergraduate Research Support Scheme},
  note={Supervised by Dr. Candida Bowtell}
}

@article{Pratt2019NQueens,
  author  = {Pratt, Kevin},
  title   = {Closed-Form Expressions for the n-Queens Problem and Related Problems},
  journal = {International Mathematics Research Notices},
  volume  = {2019},
  number  = {4},
  pages   = {1098--1107},
  year    = {2019},
  doi     = {10.1093/imrn/rnx119}
}

@article{Simkin2023NQueens,
  author  = {Simkin, Michael},
  title   = {The number of $n$-queens configurations},
  journal = {Advances in Mathematics},
  volume  = {427},
  pages   = {109127},
  year    = {2023},
  doi     = {10.1016/j.aim.2023.109127}
}

@article{Glock2022NQueens,
  title={The n-queens completion problem},
  author={Stefan Glock and David Munhá Correia and Benny Sudakov},
  journal={Research in the Mathematical Sciences},
  volume={9},
  number={41},
  year={2022},
  doi={https://doi.org/10.1007/s40687-022-00335-1}
}

@software{preusser_q27_2016,
  author = {Preußer, Thomas B.},
  title = {{q27}: 27-Queens Puzzle: Massively Parallel Enumeration and Solution Counting},
  year = {2016},
  url = {https://github.com/preusser/q27},
  note = {[computer software]},
  license = {AGPL-3.0}
}

@article{BellStevens2009,
  author       = {J. Bell and B. Stevens},
  title        = {A survey of known results and research areas for the \(n\)-queens problem},
  journal      = {Discrete Mathematics},
  volume       = {309},
  number       = {1},
  pages        = {1--31},
  year         = {2009},
  doi          = {10.1016/j.disc.2007.12.043}
}

@misc{oeisA000170,
  author       = {Sloane, N. J. A.},
  title        = {The On-Line Encyclopedia of Integer Sequences, 
                  {A000170}: Number of ways of placing $n$ nonattacking
                  queens on an $n \times n$ board},
  howpublished = {\url{https://oeis.org/A000170}},
  note         = {Accessed 6 Dec 2025.},
  year         = {2025}
}

@book{kraitchik1953mathematical,
  author    = {Kraitchik, Maurice},
  title     = {Mathematical Recreations},
  edition   = {2},
  publisher = {Dover},
  address   = {New York},
  year      = {1953}
}

@article{Pauls1874Damenproblem,
  author  = {Pauls, Emil},
  title   = {Das Maximalproblem der Damen auf dem Schachbrete},
  journal = {Deutsche Schachzeitung},
  volume  = {29},
  year    = {1874},
  pages   = {129--134, 257--267}
}
\end{document}